\documentclass[reqno]{amsart} 
\usepackage{amssymb, dsfont, fancyhdr, graphicx, color, tabularx, mathtools, hyperref,tikz,comment}

\usepackage{geometry}
\usepackage[shortlabels]{enumitem} 
\usepackage[all]{xy} 
\usepackage{ytableau}
\usepackage[capitalize,nameinlink,noabbrev,nosort]{cleveref} 
\crefname{equation}{}{}
\usepackage{cite}

\def\Imm{\operatorname{Imm}}

\newcommand{\CC}{{\mathbb C}}

\newcommand{\Mat}{\text{Mat}}
\newcommand{\res}[1]{|_{#1}} 
\newcommand{\gr}[1]{\Gamma[#1]} 
\newcommand{\gro}[1]{\Gamma(#1)} 

\newcommand{\lrangle}[1]{\langle #1 \rangle}

\newcommand{\bbox}[1]{\mathbf{B}_{#1}}

\newcommand{\multichoose}[2]{
\left.\mathchoice
  {\left(\kern-0.48em\binom{#1}{#2}\kern-0.48em\right)}
  {\big(\kern-0.30em\binom{\smash{#1}}{\smash{#2}}\kern-0.30em\big)}
  {\left(\kern-0.30em\binom{\smash{#1}}{\smash{#2}}\kern-0.30em\right)}
  {\left(\kern-0.30em\binom{\smash{#1}}{\smash{#2}}\kern-0.30em\right)}
\right.}

\makeatletter
\newcommand{\specialcell}[1]{\ifmeasuring@#1\else\omit$\displaystyle#1$\ignorespaces\fi}
\makeatother

\newtheorem{prop}{Proposition}[section] 
\newtheorem{cor}[prop]{Corollary} 
\newtheorem{lem}[prop]{Lemma}
\newtheorem{conj}[prop]{Conjecture}
\newtheorem{thm}[prop]{Theorem}

\theoremstyle{definition}
\newtheorem{defn}[prop]{Definition}

\newtheorem{ex}[prop]{Example}

\numberwithin{equation}{section} 

\title[$k$-positivity of 1324- and 2143-avoiding dual canonical basis elements]{$k$-positivity of dual canonical basis elements from 1324- and 2143-avoiding Kazhdan-Lusztig immanants}
\author{Sunita Chepuri$^*$}\thanks{$^*$University of Michigan, 2074 East Hall, 530 Church Street. Ann Arbor, MI 48109, chepuri@umich.edu}
\author{Melissa Sherman-Bennett$^\dagger$}\thanks{$^\dagger$University of Michigan, 2074 East Hall, 530 Church Street. Ann Arbor, MI 48109, msherben@umich.edu}

\begin{document}

\maketitle
\begin{abstract}
In this note, we show that certain dual canonical basis elements of $\CC[SL_m]$ are positive when evaluated on \emph{$k$-positive matrices}, matrices whose minors of size $k \times k$ and smaller are positive. Skandera showed that all dual canonical basis elements of $\CC[SL_m]$ can be written in terms of \emph{Kazhdan-Lusztig immanants}, which were introduced by Rhoades and Skandera. We focus on the basis elements which are expressed in terms of Kazhdan-Lusztig immanants indexed by 1324- and 2143-avoiding permutations. This extends previous work of the authors on Kazhdan-Lusztig immanants and uses similar tools, namely Lewis Carroll's identity (also known as the Desnanot-Jacobi identity).
\end{abstract}


\section{Introduction}\label{sec:intro}

Given a function $f:S_n \to \CC$, the \emph{immanant} associated to $f$, $\Imm_f X:\Mat_{n \times n}(\CC) \to \CC$, is the function
\begin{equation}
\Imm_f X:= \sum_{w \in S_n} f(w) ~x_{1, w(1)} \cdots x_{n, w(n)},
\end{equation}
where the $x_{i,j}$ are indeterminates.  We evaluate $\Imm_f X$ on a matrix $M=(m_{i,j})$ by specializing $x_{i,j}$ to $m_{i,j}$ for all $i,j$.

Immanants are a generalization of the determinant, where $f(w)=(-1)^{\ell(w)}$, and the permanent, where $f(w)=1$.  
Positivity properties of immanants have been studied since the early 1990's \cite{GJ,Greene,StemConj,HaiConj}.  One of the main results in this area is that when $f$ is an irreducible character of $S_n$, then $\Imm_f(X)$ is nonnegative on \emph{totally nonnegative matrices}, that is, matrices with all nonnegative minors~\cite{StemTNN}.
In this note, we will investigate positivity properties of functions closely related to \emph{Kazhdan--Lusztig immanants}, introduced by Rhoades and Skandera~\cite{RS}.

\begin{defn}
	Let $v \in S_n$. The \emph{Kazhdan-Lusztig immanant} $\Imm_v X: \Mat_{n \times n}(\CC) \to \CC$ is given by 
	\begin{equation} \label{eq:immFormula1}
	\Imm_v X:= \sum_{w \in S_n} (-1)^{\ell(w)-\ell(v)} P_{w_0w, w_0v}(1) ~x_{1, w_1} \cdots x_{n, w_n}
	\end{equation}
	where $P_{x, y}(q)$ is the Kazhdan-Lusztig polynomial associated to $x,y \in S_n$, $w_0 \in S_n$ is the longest permutation, and we write permutations $w=w_1w_2\dots w_n$ in one-line notation. (For the definition of $P_{x, y}(q)$ and their basic properties, see e.g. \cite{BB}.)
\end{defn}

Our interest in Kazhdan--Lusztig immanants stems from their connection to the dual canonical basis of $\CC[SL_m]$. Using work of Du~\cite{Du}, Skandera~\cite{Skan} showed that the dual canonical basis elements of $\CC[SL_m]$ are exactly Kazhdan--Lusztig immanants evaluated on matrices of indeterminates with repeated rows and columns.

Let $X=(x_{ij})$ be the $m \times m$ matrix of variables $x_{ij}$ and let $\multichoose{[m]}{n}$ denote the set of $n$-element multisets of $[m]:=\{1, \dots, m\}$.  For $R, C\in \multichoose{[m]}{n}$ with $R=\{r_1 \leq \cdots \leq r_n\}$ and $C=\{c_1\leq \cdots \leq c_n\}$, we write $X(R,C)$ to denote the matrix $(x_{r_i, c_j})_{i, j=1}^n$ (see Definition~\ref{defn:RC}).

\begin{prop}[\protect{\cite[Theorem 2.1]{Skan}}] \label{prop:canonicalImm}
The dual canonical basis of $\CC[SL_m]$ consists of the nonzero elements of the following set:
\[\left \{\Imm_v X(R, C): v \in S_n \text{ for some } n \in \mathbb{N} \text{ and } R, C \in \multichoose{[m]}{n}\right \}.
\]

\end{prop}


The positivity properties of dual canonical basis elements have been of interest essentially since their definition, and are closely related to the study of total positivity.  In 1994, Lusztig~\cite{LusTPGr} defined the totally positive part $G_{>0}$ of any reductive group $G$.  He also showed that all elements of the dual canonical basis of $\mathcal{O}(G)$ are positive on $G_{>0}$.  Fomin and Zelevinsky~\cite{FZTNNSemisimple} later proved that for semisimple groups, $G_{>0}$ is precisely the subset of $G$ where all \emph{generalized minors} are positive. Generalized minors are dual canonical basis elements corresponding to the fundamental weights of $G$ and their images under Weyl group action.

Here, we study signs of dual canonical basis elements on a natural generalization of $G_{>0}$. Let $S$ be some subset of generalized minors and $G_{>0}^S$ the subset of $G$ where all elements of $S$ are positive. Which dual canonical basis elements are positive on all elements of $G_{>0}^S$? In this note, we consider the case where $G=SL_m$ and $S$ consists of the generalized minors corresponding to the first $k$ fundamental weights and their images under the Weyl group action.  In this situation, $G_{>0}^S$ is the set of \emph{$k$-positive matrices}, matrices where all minors of size $k$ and smaller are positive.  Cluster algebra structures, topology, and variation diminishing properties of these matrices have been previously studied in~\cite{BCM,CKST,Choud1,Choud2}.

We call a matrix functional \emph{$k$-positive} if it is positive when evaluated on all $k$-positive matrices.  Our main result is as follows:

\begin{thm}  \label{thm:main}
	Let $v \in S_n$ be $1324$-, $2143$-avoiding and suppose that for all $i<j$ with $v_i<v_j$, we have $j-i \leq k$ or $v_j-v_i \leq k$. Let $R, C \in \multichoose{[m]}{n}$. Then $\Imm_v X(R,C)$ is identically zero or it is $k$-positive.
\end{thm}

We also characterize precisely when the functions $\Imm_v X(R, C)$ appearing in \cref{thm:main} are identically zero (see Theorem~\ref{thm:main-sq}).

Theorem~\ref{thm:main} extends the results of~\cite{CSB}, in which we showed the function $\Imm_v X([m],[m])$ is $k$-positive under the assumptions of \cref{thm:main}. Our techniques here are similar to~\cite{CSB}. Note that Theorem~\ref{thm:main} does not follow from \cite[Theorem 1.4]{CSB} because for $M$ $k$-positive, $M(R, C)$ is $k$-nonnegative rather than $k$-positive.


Rephrasing Theorem~\ref{thm:main} in terms of dual canonical basis elements, we have the following corollary.

\begin{cor}\label{cor:main}
Let $F(X)=\Imm_v X(R,C)$ be an element of the dual canonical basis of $\CC[SL_m]$. Suppose $v$ is $1324$-, $2143$-avoiding and for all $i<j$ with $v_i<v_j$, we have $j-i \leq k$ or $v_j-v_i \leq k$. Then $F(X)$ is $k$-positive.
\end{cor}

The paper is organized as follows. Section~\ref{sec:prelim} gives background on the objects we will be using to prove Theorem~\ref{thm:main}. It includes several useful lemmas proven in~\cite{CSB}.  Section~\ref{sec:proof} contains the proof of Theorem~\ref{thm:main}.  We conclude with a few thoughts on future directions in Section~\ref{sec:future}.

\section{Background}\label{sec:prelim}
In an abuse of notation, we frequently drop curly braces around sets appearing in subscripts and superscripts.

\subsection{Background on 1324 and 2143-avoiding Kazhdan-Lusztig immanants}

For integers $i\leq j$, let $[i, j]:=\{i, i+1, \dots, j-1, j\}$. We abbreviate $[1, n]$ as $[n]$. For $v \in S_n$, we write $v_i$ or $v(i)$ for the image of $i$ under $v$. We use the notation $<$ for both the usual order on $[n]$ and the Bruhat order on $S_n$; it is clear from context which is meant. To discuss non-inversions of a permutation $v$, we'll write $\lrangle{i, j}$ to avoid confusion with a matrix index or point in the plane. In the notation $\lrangle{i, j}$, we always assume $i<j$. We use the notation $\multichoose{[m]}{n}$ for the collection of $n$-element multi-sets of $[m]$. We always list the elements of a multiset in increasing order.

We are concerned with two notions of positivity, one for matrices and one for immanants.

\begin{defn}
Let $k\geq 1$. A matrix $M \in \Mat_{n \times n}(\CC)$ is \emph{$k$-positive} if all minors of size at most $k$ are positive. 

An immanant $\Imm_f(X): \Mat_{n \times n}(\CC) \to \CC$ is \emph{$k$-positive} if it is positive on all $k$-positive matrices.
\end{defn}

Note that $k$-positive matrices have positive $1 \times 1$ minors, i.e. entries, and so are real matrices. 

\begin{ex}\label{ex:k-pos} The matrix
\[ M=\begin{bmatrix}
    22 & 18 & 6 & 3 \\
    8 & 7 & 3 & 2 \\
    2 & 2 & 1 & 2 \\
    1 & 2 & 2 & 6
    \end{bmatrix}
    \]
     is $2$-positive but the upper left $3\times 3$ submatrix has negative determinant, so is not $3$-positive or $4$-positive (totally positive). 
\end{ex}

Our results on $k$-positivity of Kazhdan-Lusztig immanants involve pattern avoidance.

\begin{defn}\label{defn:patternAvoidance}
	Let $v \in S_n$, and let $w\in S_m$. Suppose $v=v_1\cdots v_n$ and $w=w_1 \cdots w_m$ in one-line notation. The pattern $w_1 \cdots w_m$ \emph{occurs} in $v$ if there exists $1\leq i_1< \dots <i_m\leq n$ such that $v_{i_1} \cdots v_{i_m}$ are in the same relative order as $w_1 \cdots w_m$. Additionally, $v$ \emph{avoids} the pattern $w_1 \cdots w_m$ if it does not occur in $v$.
\end{defn}

Certain Kazdhan-Lusztig immanants have a very simple determinantal formula, which involves the \emph{graph} of an interval.

\begin{defn}\label{defn:graph}
For $v \in S_n$, the \emph{graph} of $v$, denoted $\gro{v}$, refers to its graph as a function. That is, $\gro{v}:=\{(1, v_1), \dots, (n, v_n)\}$.
For $v, w \in S_n$, the graph of the Bruhat interval $[v, w]$ is the subset of $[n]^2$ defined as $\gr{v,w}:=\{(i, u_i): u \in [v, w], i=1, \dots, n\}$.
\end{defn} 

We think of an element $(i,j) \in \gr{v,w}$ as a point in row $i$ and column $j$ of an $n \times n$ grid, indexed so that row indices increase going down and column indices increase going right (see \cref{ex:graph}). A \emph{square} or \emph{square region} in $\gr{v,w}$ is a subset of $\gr{v,w}$ which forms a square when drawn in the grid.

We will also need the following notion on matrices. 

\begin{defn}
Let $P \subset [n]^2$ and let $M=(m_{ij})$ be an $n \times n$ matrix. The \emph{restriction} of $M$ to $P$, denoted $M \res{P}$, is the matrix with entries
\[m'_{ij}=\begin{cases} m_{ij} & (i,j) \in P\\
0 & \text{ else}.
\end{cases}
\]
\end{defn}

\begin{ex}\label{ex:graph}
Consider $v=2413$ in $S_4$. We have $[v, w_0]=\{2413, 4213, 3412, 2431, 4312, 4231, 3421\}$, and so $\gr{v,w_0}$ is as follows.
\begin{center}
   \includegraphics[height=0.15\textheight]{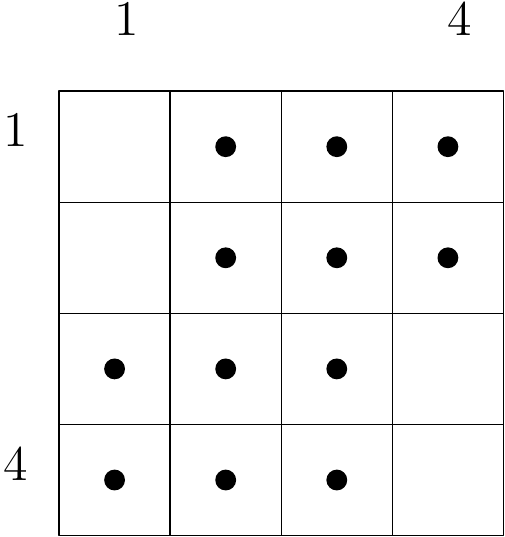}
\end{center}
If $M$ is the matrix from Example~\ref{ex:k-pos}, then  \[ M \res{\gr{v,w_0}}=\begin{bmatrix}
    0 & 18 & 6 & 3 \\
    0 & 7 & 3 & 2 \\
    2 & 2 & 1 & 0 \\
    1 & 2 & 2 & 0
    \end{bmatrix}.
    \]
Note that $v$ avoids patterns 1324 and 2143.
\end{ex}

We can now state a simple determinantal formula for certain Kazhdan-Lusztig elements. This follows from results of \cite{Sjo}.

\begin{prop}[\protect{\cite[Corollary 3.6]{CSB}}] \label{prop:immDet} Let $v \in S_n$ avoid $1324$ and $2143$. Then

\begin{equation}\label{eq:immDet}
\Imm_v(X)=(-1)^{\ell(v)} \det(X\res{\gr{v, w_0}}).
\end{equation}

\end{prop}

Using \cref{prop:immDet}, we can similarly obtain a simple determinantal formula for certain dual canonical basis elements of $\CC[SL_m]$. Recall from \cref{prop:canonicalImm} that every dual canonical basis element can be expressed as a Kazhdan-Lusztig immanant evaluated on a matrix of indeterminants with repeated rows and columns.

\begin{defn}\label{defn:RC}
Let $R=\{r_1 \leq r_2 \leq \dots \leq r_n\}$ and $C=\{c_1 \leq c_2 \leq \dots \leq c_n\}$ be elements of $\multichoose{[m]}{n}$ and let $M=(m_{ij})$ be an $m \times m$ matrix. We denote by $M(R,C)$ the matrix with $(i,j)$-entry equal to $m_{r_i, c_j}$. We call $r_i$ the \emph{label} of row $i$; similarly, $c_j$ is the label of column $j$.  We view $X(R, C)$ as a function from $\Mat_{m \times m}(\CC)$ to $\Mat_{n \times n}(\CC)$, which takes $M$ to $M(R, C)$.

\end{defn}

Note that our convention is always to list multisets in weakly increasing order, so the row and column labels of $X(R, C)$ are weakly increasing.

\begin{ex}
Let $R=\{1,1,3\}$ and $C=\{2,3,4\}$.  Then 
$$X(R,C)=\begin{bmatrix}
x_{12} & x_{13} & x_{14} \\ 
x_{12} & x_{13} & x_{14} \\ 
x_{32} & x_{33} & x_{34}
\end{bmatrix}.$$
If $M$ is the matrix from Example~\ref{ex:k-pos}, then 
$$M(R,C)=\begin{bmatrix}
18 & 6 & 3 \\ 
18 & 6 & 3 \\ 
2 & 1 & 2
\end{bmatrix}.$$
\end{ex}

We will focus on the dual canonical basis elements $\Imm_v X(R,C)$ where $v$ is 1324- and 2143-avoiding. \cref{prop:immDet} immediately gives a determinantal formula for these immanants.

\begin{lem} \label{lem:dualCanonicalDet} Let $R, C \in \multichoose{[m]}{n}$ and let $v \in S_n$ be $1324$- and $2143$-avoiding. Then 
\begin{equation} \label{eq:dualCanonicalDet} \Imm_v X(R, C)= (-1)^{\ell(v)} \det X(R, C)\res{\gr{v, w_0}}.
\end{equation}
\end{lem}

We are interested in the sign of $\Imm_v X(R, C)$ on $k$-positive matrices, so long as $\Imm_v X(R, C)$ is not identically zero. Clearly, the function in \eqref{eq:dualCanonicalDet} is identically zero when the matrix $X(R, C) \res{\gr{v, w_0}}$ has two identical rows or columns. We make the following definitions to discuss this situation.

\begin{defn} 
Let $P \subseteq [n]^2$. The \emph{support} of row $r$ of $P$ is the set of columns $c \in [n]$ such that $(r, c)\in P$. The support of a column is defined analogously.
\end{defn}

\begin{defn}
 Let $P \subseteq [n]^2$, and let $R, C \in \multichoose{[m]}{n}$. Then $P$ is \emph{$(R,C)$-admissible} if no two rows or columns with the same labels have the same support.
\end{defn}

\begin{ex}
Let $P=\gr{v,w_0}$ where $v=2413$, as in Example~\ref{ex:graph}. Rows 1 and 2 have support $\{2, 3, 4\}$ and rows 3 and 4 have support $\{1, 2, 3\}$. Column 1 has support $\{3,4\}$, columns 2 and 3 have support $\{1, 2, 3, 4\}$, and column 4 has support $\{1, 2\}$. This means $P$ is $(R,C)$-admissible if and only if $r_1\neq r_2, r_3\neq r_4$, and $c_2\neq c_3$. For example, let $A=\{1,2,2,3\}$ and $B=\{1,2,3,3\}$. Then $P$ is $(A, B)$-admissible but, since $a_2=a_3=2$, $P$ is not $(A, A)$-admissible.
\end{ex}

For $v$ avoiding 1324 and 2143, $\Imm_v X(R, C)$ is identically zero if $\gr{v, w_0}$ is not $(R, C)$-admissible. In the subsequent sections, we will show the converse holds as well (see \cref{thm:detSign}).

Finally, we introduce some notation that will be useful in proofs.

For $I \in \binom{[n]}{k}$, define $\delta_I:[n] \setminus I \to [n-k]$ as
\[ 
\delta_I(j):= j-|\{i \in I: i<j\}|
\]
That is, $\delta_I$ is the unique order-preserving map from $[n] \setminus I$ to $[n-k]$.

\begin{defn}
For $I, J \in \binom{[n]}{k}$ and $P \subseteq [n]^2$, let $P^J_I\subseteq[n-k]\times [n-k]$ be $P$ with rows $I$ and columns $J$ deleted. That is, $P^J_I=\{(\delta_I(a), \delta_J(b)): (a, b) \in P\}$. The labels of rows and columns are preserved under deletions; to be more precise, if $R=\{r_1 \leq \cdots \leq r_n\}$ is the multiset of row labels of $P$, the multiset of row labels of $P^J_I$ is $\{r'_1 \leq \cdots \leq  r'_{n-k}\}$ where $r'_j= r_{\delta_I^{-1}(j)}$.
\end{defn}

\subsection{Combinatorics of graphs of upper intervals}

We will now take a closer look at the graphs $\gr{v,w_0}$ that appear in \Cref{lem:dualCanonicalDet}. We begin by giving an alternate definition for $\gr{v,w_0}$.

\begin{defn}
Let $v \in S_n$ and $(i, j) \in [n]^2 \setminus \gro{v}$. Then $(i, j)$ is \emph{sandwiched} by a non-inversion $\lrangle{k, l}$ if $k \leq i \leq l$ and $v_k \leq j \leq v_l$. We also say $\lrangle{k, l}$ \emph{sandwiches} $(i, j)$.
\end{defn}

In other words, $(i, j)$ is sandwiched by $\lrangle{k, l}$ if and only if $(i, j)\in[n]^2$ lies inside the rectangle with diagonal corners $(k, v_k)$ and $(l, v_l)$.

\begin{lem}[\protect{\cite[Lemma 3.4]{CSB}}] \label{lem:graphCharacterization}
Let $v\in S_n$. Then $\gr{v,w_0}=\gro{v} \cup \{(i, j): (i, j)$  is sandwiched by a non-inversion of $v\}$.
\end{lem}

Using this alternate characterization, one can translate the assumptions of \cref{thm:main} into a condition on $\gr{v, w_0}$.

\begin{lem}[\protect{\cite[Lemma 4.1]{CSB}}]\label{lem:sq-inversions}
Let $v\in S_n$. The graph $\gr{v,w_0}$ has a square of size $k+1$ if and only if for some non-inversion $\lrangle{i, j}$ of $v$, we have $j-i \geq k$ and $v_j-v_i \geq k$.
\end{lem}

We now introduce some notation and a proposition that we will need to prove our main result.

\begin{defn}\label{defn:bounding-box-anti}
Let $v \in S_n$.  Define $\bbox{i, v_i}$ to be the square region of $[n]^2$ with corners $(i,v_i),\ (i, n-i+1),\ (n-v_i+1, v_i)$ and $(n-v_i+1,n-i+1)$.  In other words, $\bbox{i, v_i}$ is the square region of $[n]^2$ with one corner at $(i,v_i)$ and two corners on the antidiagonal of $[n]^2$. We say $\bbox{i, v_i}$ is a \emph{bounding box} of $\gr{v,w_0}$ if there does not exist some $j$ such that $\bbox{i,v_i}\subsetneq \bbox{j,v_j}$. If $\bbox{i, v_i}$ is a bounding box of $\gr{v,w_0}$, we call $(i, v_i)$ a \emph{spanning corner} of $\gr{v,w_0}$. 
(See \cref{fig:boundingBoxEx} for an example.)
\end{defn}

\begin{figure}
    \centering
    \includegraphics[height=0.35\textheight]{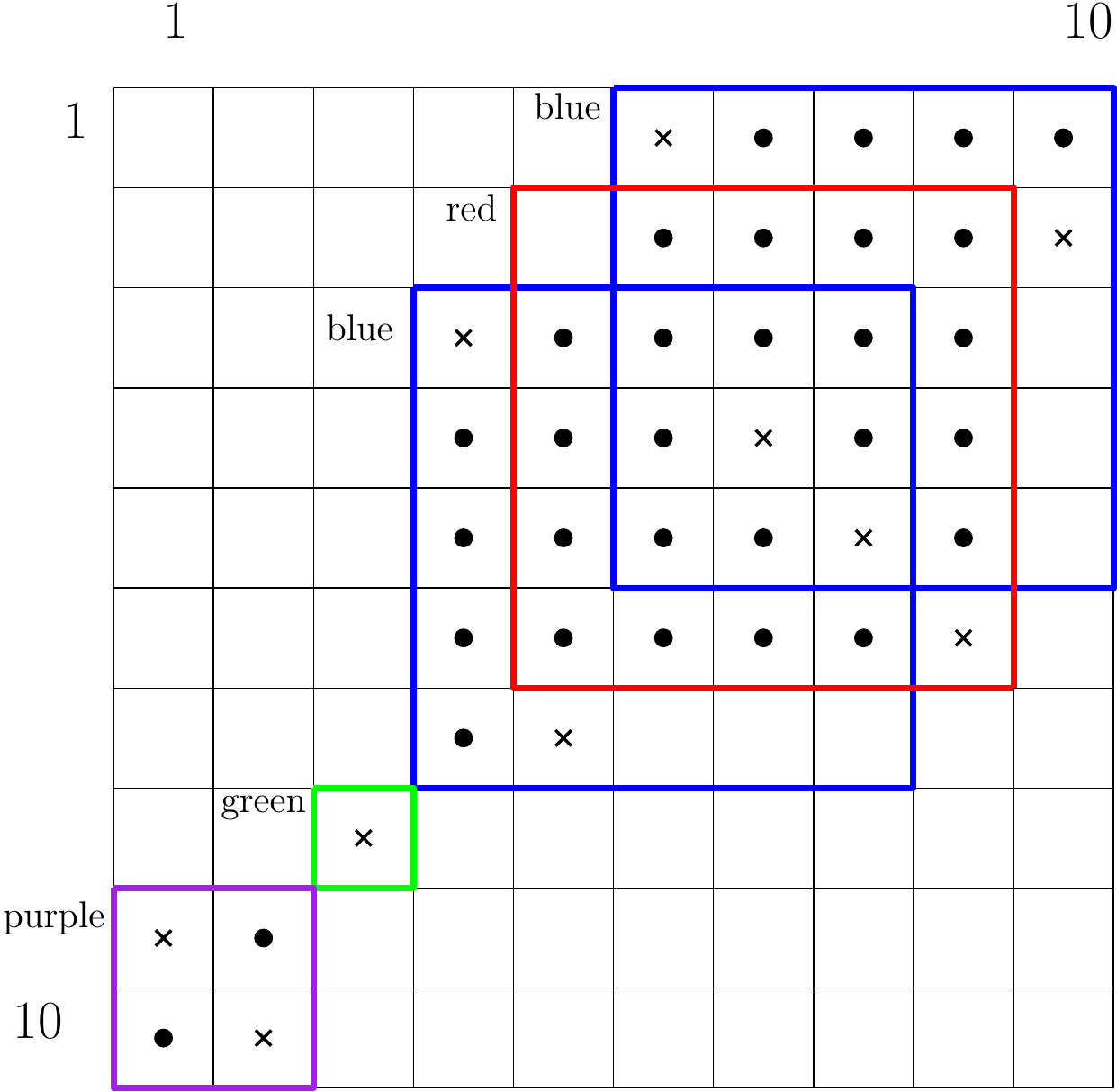}
    \caption{An example of $\gr{v,w_0}$, with $v=6~10~4~7~8~9~3~1~2$. The bounding boxes are blue, red, blue, green, and purple, listed in the order of their northmost row. The spanning corners of $\gr{v, w_0}$ are $(1, 6)$, $(3, 4)$, $(6, 9)$, $(8, 3)$, $(9, 1)$, and $(10, 2)$.}
    \label{fig:boundingBoxEx}
\end{figure}

The name ``bounding boxes" comes from the following lemma.

\begin{lem}[\protect{\cite[Lemma 4.12]{CSB}}] \label{lem:boundingboxes}
Let $v \in S_n$. Then
\[ \gr{v,w_0} \subseteq \bigcup_{(i, v_i) \in S} \bbox{i, v_i}.
\]
\end{lem}

We also color the bounding boxes.

\begin{defn}
 A bounding box $\bbox{i,v_i}$ is said to be \emph{red} if $(i,v_i)$ is below the antidiagonal, \emph{green} if $(i,v_i)$ is on the antidiagonal, and \emph{blue} if $(i,v_i)$ is above the antidiagonal.  If $\bbox{i,v_i}$ and $\bbox{n-v_i+1,n-i+1}$ are both bounding boxes, then $\bbox{i,v_i}=\bbox{n-v_i+1,n-i+1}$ is both red and blue.  We say such a box is \emph{purple}. (See \cref{fig:boundingBoxEx} for an example.)
\end{defn}

\begin{prop}[\protect{\cite[Proposition 4.14]{CSB}}] \label{prop:alternatingBoxesAnti}
Suppose $v \in S_n$ avoids $2143$ and $w_0v$ is not contained in a maximal parabolic subgroup of $S_n$. Order the bounding boxes of $\gr{v,w_0}$ by the row of the northwest corner. If $\gr{v,w_0}$ has more than one bounding box, then they alternate between blue and red and there are no purple bounding boxes.
\end{prop}

\section{Positivity of basis elements} \label{sec:proof}

In this section, we prove our main result.

\begin{thm}\label{thm:main-sq}
Let $R, C \in \multichoose{[m]}{n}$, let $v \in S_n$ be $1324$-, $2143$-avoiding and suppose that the largest square region in $\gr{v,w_0}$ has size at most $k$.  If $\gr{v, w_0}$ is not $(R, C)$-admissible, then $\Imm_v X(R,C)$ is identically zero. Otherwise, $\Imm_v X(R,C)$ is $k$-positive.
\end{thm}

Theorem~\ref{thm:main} easily follows from Theorem~\ref{thm:main-sq}, using Lemma~\ref{lem:sq-inversions}.

Our proofs rely heavily on Lewis Carroll's identity.

\begin{prop}[Lewis Carroll's Identity]\label{prop:lc-id}
If $M$ is an $n\times n$ square matrix and $M_A^B$ is $M$ with the rows indexed by $A \subset [n]$ and columns indexed by $B \subset [n]$ removed, then 
$$\det(M)\det(M_{a,a'}^{b,b'})=\det(M_a^b)\det(M_{a'}^{b'})-\det(M_a^{b'})\det(M_{a'}^b),$$
where $1\leq a<a'\leq n$ and $1\leq b<b'\leq n$.
\end{prop}

\subsection{Young diagram case}\label{sec:young}

We first consider the case where $\gr{v, w_0}$ is a Young diagram or the complement of a Young diagram (using English notation). Recall that the \emph{Durfee square} of a Young diagram $\lambda$ is the largest square contained in $\lambda$.

\begin{prop} \label{prop:partitiondet}
Let $\lambda \subseteq n^n$ be a Young diagram with Durfee square of size at most $k$ and $\mu:=n^n/\lambda$. Let $M$ be a $m\times m$ $k$-positive matrix and $R,C\in \multichoose{[m]}{n}$. Then 

\[(-1)^{|\mu|}\det M(R,C)\res{\lambda} \geq 0
\]
and equality holds only if $(n, n-1, \dots, 1) \nsubseteq \lambda$ or if $\lambda$ is not $(R,C)$-admissible.
\end{prop}

\begin{proof}
Let $A=M(R,C)\res{\lambda}=\{a_{ij}\}$. For $\sigma \in S_n$, let $a_{\sigma}:=a_{1, \sigma(1)}\cdots a_{n, \sigma(n)}$.  If $(n, n-1, \dots, 1) \nsubseteq \lambda$ then there is some $1\leq j\leq n$ where $\lambda_{n-j+1}<j$.  Thus boxes in $\lambda$ in the last $j$ rows are in the southwest most $j\times (j-1)$ rectangle.  This means that for every $\sigma$, $a_\sigma$ contains some zero entry, so $\det(A)=0$.  It's clear that if $\lambda$ is not $(R,C)$-admissable then $\det(A)=0$.

Now we will assume that $(n, n-1, \dots,1) \subseteq \lambda$ and that $\lambda$ is $(R,C)$-admissible. We proceed by induction on $n$ to show that $\det(A)$ has sign $(-1)^{|\mu|}$. The base cases for $n=1, 2$ are easy to check.

Let $a=\max\{i\ |\ \lambda_i=n\}$ and $b=\lambda_n=\max\{j\ |\ \lambda'_j=n\}$ where $\lambda'$ denotes the transpose of $\lambda$.  In other words, $a$ is the last row in $\lambda$ with $n$ boxes and $b$ is the last column in $\lambda$ with $n$ boxes.  From Lewis Carroll's identity, we have that 
\begin{equation}\label{eq:CarrolPartition}
\det(A) \det(A_{a, n}^{b, n})=\det(A_a^b)\det(A_n^n)-\det(A_a^n)\det(A_n^b).
\end{equation}

Let's see what we know about the signs of these determinants using our inductive hypothesis.  Say $I:=\{i_1 < \cdots < i_k\}$ and $J:=\{j_1< \cdots < j_k\}$, and let $\lambda_I^J$ denote the Young diagram obtained from $\lambda$ by removing rows indexed by $I$ and columns indexed by $J$. Note that $$A_I^J=M(R,C)_I^J\res{\lambda_I^J}=M(R\setminus\{r_{i_1}, \dots, r_{i_k}\},C\setminus\{c_{j_1}, \dots, c_{j_k}\})\res{\lambda_I^J}.$$ Also,
$\lambda_I^J$ has Durfee square of size at most $k$.  So we can use the inductive hypothesis to compute the signs of all of the determinants in \eqref{eq:CarrolPartition} other than $\det(A)$.

Let's consider which determinants in \eqref{eq:CarrolPartition} are zero. The shape $\lambda_{a, n}^{b, n}$ contains the staircase $(n-2,\dots,1)$ and the shapes $\lambda_n^n,\lambda_a^n$,  and $\lambda_n^b$ contain the staircase  $(n-1,\dots,1)$.  However, $\lambda_a^b$ may not contain the staircase $(n-1,\dots,1)$ (e.g. consider $\lambda=(3,3,1)$), so $\det A_a^b$ may be zero.  Now we need to determine when $\lambda_I^J$ is $(R\setminus\{r_{i_1}, \dots, r_{i_k}\},C\setminus\{c_{j_1}, \dots, c_{j_k}\})$-admissible.  Consider $A_{a, n}^{b, n}$ and pick two row indices $p, q \notin \{a, n\}$ with $p<q$ and $r_p=r_q$. Because $\lambda$ is $(R, C)$-admissible, rows $p, q$ have different support, so $\lambda_p>\lambda_q$. Further, because $R$ is listed in weakly increasing order, $p>a$. We would like to argue that rows $p':=\delta_{a, n}(p)$ and $q':=\delta_{a, n}(q)$ of $A_{a, n}^{b, n}$ have distinct support. Since $p>a$, we have $(\lambda_{a, n}^{b, n})_{p'}=\lambda_p-1$ and $(\lambda_{a, n}^{b, n})_{q'}=\lambda_q-1$, so $(\lambda_{a, n}^{b, n})_{p'}>(\lambda_{a, n}^{b, n})_{q'}$.  An analogous argument shows that columns of $A_{a, n}^{b, n}$ with the same index have different support.  Similarly, $A_a^b,A_a^n$, and $A_n^b$ are $(S, D)$-admissible for the appropriate $S, D$.  On the other hand, $A_n^n$ may not be (consider $R=(1,1,2)$, $C=(1,2,3)$, $\lambda=(3,2,1)$, for example).

Taking all of this together we find that the $\det(A_a^b)\det(A_n^n)$ term in \eqref{eq:CarrolPartition} may be zero but that $\det(A_{a, n}^{b, n})$ and $\det(A_a^n)\det(A_n^b)$ are always nonzero.  By induction, $\det(A_{a, n}^{b, n})$ has sign $(-1)^{|\mu|+a+b+1}$ and $\det(A_a^n)\det(A_n^b)$ has sign $(-1)^{a+b}$.  If $\det(A_a^b)\det(A_n^n)$ is nonzero it has sign $(-1)^{a+b+1}$.  Thus, $\det(A_a^b)\det(A_n^n)-\det(A_a^n)\det(A_n^b)$ always has sign $(-1)^{a+b+1}$ and $\det(A)$ is always nonzero with sign $(-1)^{|\mu|}$.
\end{proof}

\begin{cor}\label{cor:partitioncompdet}
Let $\lambda \subseteq n^n$ be a Young diagram and let $\mu:= n^n /\lambda$. Suppose $\mu$ has Durfee square of size at most $k$, $M$ is a $k$-positive $m \times m$ matrix, and $R,C\in \multichoose{[m]}{n}$. Then

\[ (-1)^{|\lambda|}\det M(R,C)\res\mu \geq 0
\]
and equality holds if and only if $ (n^n/ (n-1, n-2, \dots, 1, 0) ) \nsubseteq \mu$ (or equivalently, $\lambda \nsubseteq (n-1, n-2, \dots, 1, 0)$) or if $\mu$ is not $(R,C)$-admissible.
\end{cor}

\begin{proof} Let $\dot{w_0}$ denote the matrix with ones on the antidiagonal and zeros elsewhere. For a multiset $J=\{j_1 \leq \cdots \leq j_n\}$, let $\overline{J}:=\{\overline{j}_1 \leq \cdots \leq \overline{j}_n\}$ where $\overline{j}_i:=n+1-j_{n+1-i}$.

Let $M'$ be the antidiagonal transpose of $M$; in symbols, $M'=\dot{w_0} M^T \dot{w_0}$. Taking antidiagonal transpose does not effect the determinant, so $M'$ is also $k$-positive.

If we transpose $M(R,C)\res{\mu}$ across the antidiagonal, we obtain the matrix $$N:=M'(\overline{C}, \overline{R})\res{\nu},$$ where $\nu$ is the Young diagram obtained from the skew-shape $\mu$ by reflecting across the antidiagonal. Applying \cref{prop:partitiondet}, we have that $\det N$ has sign $|\lambda|$ if $\nu$ is $(\overline{C}, \overline{R})$-admissible and is zero otherwise. It is not hard to check that $\nu$ is $(\overline{C}, \overline{R})$-admissible if and only if $\mu$ is $(R, C)$-admissible.
\end{proof}

We can use \cref{prop:immDet} to rewrite \cref{prop:partitiondet} and \cref{cor:partitioncompdet} in terms of immanants.

\begin{cor}\label{cor:partitionimm}
Let $v \in S_n$ avoid $1324$ and $2143$. Suppose $\gr{v,w_0}$ is a Young diagram $\lambda$ with Durfee square of size at  most $k$. If $M$ is a $k$-positive $m \times m$ matrix and $R,C\in \multichoose{[m]}{n}$ such that $\lambda$ is $(R,C)$-admissible, then $\Imm_v M(R,C)>0$.
\end{cor}

\begin{proof} Note that $\gro{w_0} \subseteq \gr{v, w_0}$ implies $\lambda$ contains the partition $(n, n-1, \dots, 1)$.  So, by \cref{prop:partitiondet}, we know that $(-1)^{|\mu|}\det M(R,C)\res{\gr{v,w_0}}>0$ where $\mu=n^n/\lambda$.

Notice that if a box of $\mu$ is in row $r$ and column  $c$  then $v(r)<c$ and $v^{-1}(c)<r$. This means that $(v^{-1}(c),r)$ is an inversion. If $(a, b)$ is an inversion of $v$ and the box in row $b$ and column $v(a)$ is not in $\mu$, then $(b,v(a))$ is sandwiched by some non-inversion $\langle a,j\rangle$ for some $j$. But then $1 ~v(a)~v(b)~v(j)$ is an occurrence of the pattern 1324, a contradiction.  So $(b,v(a))$ is in $\mu$.  This means boxes in $\mu$ are in bijection with inversions of $v$ and $(-1)^{\ell(v)}\det M(R,C)\res{\gr{v,w_0}}=(-1)^{|\mu|}\det M(R,C)\res{\gr{v,w_0}}>0$.  By \cref{prop:immDet}, this means $\Imm_v M(R,C)>0$.
\end{proof}

\begin{cor}\label{cor:partitioncompimm}
Let $v \in S_n$ avoid $1324$ and $2143$. Suppose $\gr{v,w_0}$ is $\lambda=n^n/\mu$ for some partition $\mu$ and the largest square in $\lambda$ is of size at most $k$. If $M$ is a $k$-positive $m \times m$ matrix and $R,C\in \multichoose{[m]}{n}$ such that $\lambda$ is $(R,C)$-admissible, then $\Imm_v M(R,C)>0$.
\end{cor}

\begin{proof}
Note that $\gro{w_0} \subseteq \gr{v,w_0}$ implies $\lambda$ contains the partition $(n^n/(n-1, n-2, \dots, 1,0))$. So, by \cref{cor:partitioncompdet}, we know that $(-1)^{|\mu|}\det M(R,C)\res{\gr{v,w_0}}>0$.

As in the proof of \cref{cor:partitionimm}, there is a bijection between boxes of $\mu$ and inversions of $v$.  So, we know $(-1)^{\ell(v)}\det M(R,C)\res{\gr{v,w_0}}=(-1)^{|\mu|}\det M(R,C)\res{\gr{v,w_0}}>0$.  By \cref{prop:immDet}, this means $\Imm_v M(R,C)>0$.
\end{proof}

\subsection{General Case}\label{sec:general}

The following proposition will allow us to restrict to permutations that are not elements of a maximal parabolic subgroup of $S_n$.  To state the lemma we temporarily denote the longest permutation in $S_j$ by $w_{(j)}$.


\begin{prop}[\protect{\cite[Corollary 4.9]{CSB}}] \label{prop:ImmMultipleBlock}
Suppose $v \in S_n$ is $1324$-, $2143$-avoiding and $\gr{v, w_0}$ is block-antidiagonal. Let $v_1 \in S_j$ and $v_2 \in S_{n-j}$ be permutations such that the upper-right antidiagonal block of $\gr{v, w_0}$ is equal to $\gr{v_1,w_{(j)}}$ and the other antidiagonal block is equal to $\gr{v_2,w_{(n-j)}}$. Then 
\[\Imm_v M=\Imm_{v_1} M([j], [n-j+1, n])\cdot \Imm_{v_2}M([j+1, n], [n-j]).
\]
\end{prop}

\begin{figure}
    \centering
    \includegraphics[height=0.3\textheight]{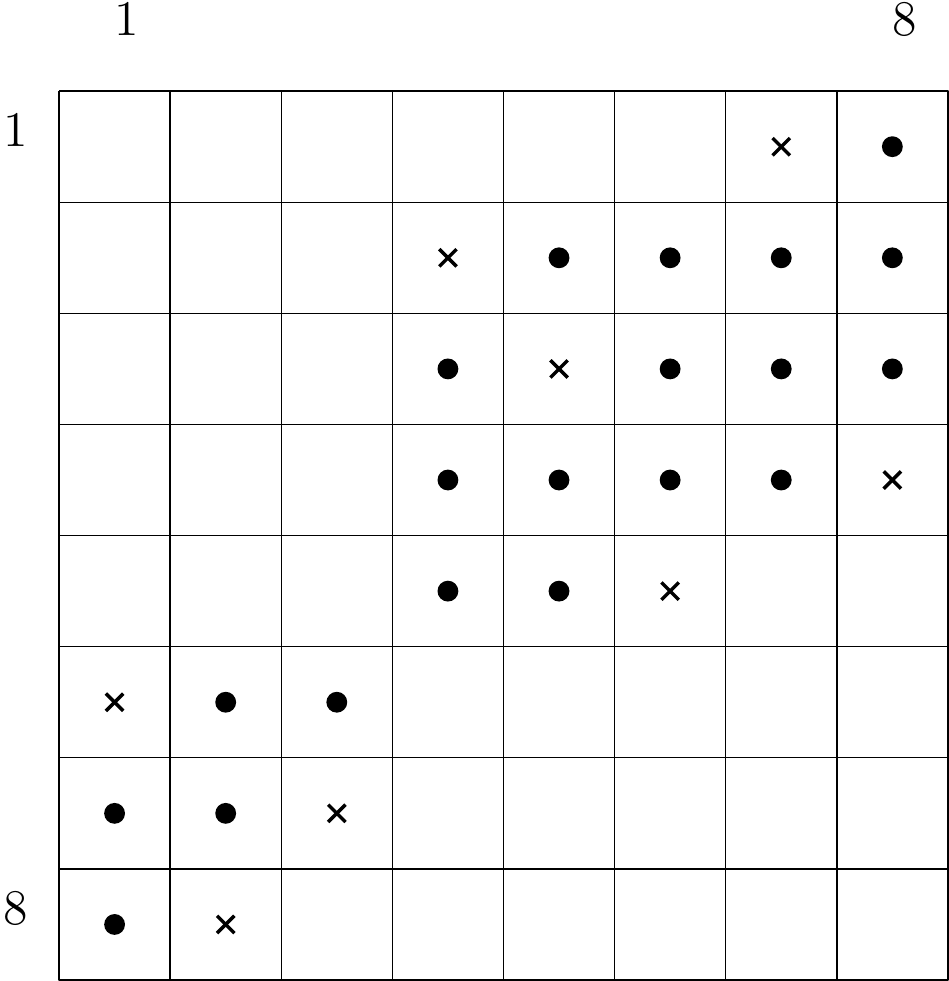}
    \caption{An example where $\gr{v,w_0}$ is block-antidiagonal. Here, $v=74586132$. In the notation of \cref{prop:ImmMultipleBlock}, $j=3$, $v_1=41253$, and $v_2=132$.}
    \label{fig:blockAnti}
\end{figure}

See Figure~\ref{fig:blockAnti} for an example illustrating a block-antidiagonal $\gr{v, w_0}$ and the notation of Proposition~\ref{prop:ImmMultipleBlock}.

To analyze the determinants appearing in Lewis Carroll's identity for $\det X(R, C)\res{\gr{v, w_0}}$, we will use the following two propositions. 

\begin{prop} \label{prop:deleteDotDetAnti}

Let $v\in S_n$ be $2143$- and $1324$-avoiding, and choose $i\in [n]$. Let $x \in S_{n-1}$ be the permutation $x: \delta_i(j) \mapsto \delta_{v_i}(v_j)$ (that is, $x$ is obtained from $v$ by deleting $v_i$ from $v$ in one-line notation and shifting the remaining numbers appropriately). Then 
\begin{enumerate}
  \item $\gr{x, w_0}=(\gr{v,w_0} \setminus \{(p, q): (p, q) \text{ is sandwiched only by a non-inversion involving }i\})_i^{v_i}$;
  
    \item If $(i, v_i)$ is not a spanning corner of $\gr{v,w_0}$, then $\gr{x,w_0}=\gr{v,w_0}_i^{v_i}$.
  
  \item For all $i$, $\det(M\res{\gr{x,w_0}})=\det(M\res{\gr{v,w_0}_{i}^{v_i}}).$
\end{enumerate}

\end{prop}

\begin{proof}
Statement (1) follows from \Cref{lem:graphCharacterization}.  Statements (2) and (3) are Proposition 4.17 from~\cite{CSB}.
\end{proof}

\begin{prop}\label{prop:sgn}
Let $v\in S_n$ be $1324$- and $2143$-avoiding such that the last bounding box of $\gr{v,w_0}$ is $\bbox{n, v_n}$, and the second to last box is $\bbox{a, v_a}$ for some $a<n$ with $1<v_a<v_n$.  Let $b=v^{-1}(1)$ and $d=v_a$.  Suppose $\det (M(R, C)\res{\gr{v, w_0}})^1_b\cdot\det (M(R, C)\res{\gr{v, w_0}})^d_a$ is nonzero and has sign $\sigma$.  Then
\begin{enumerate}
    \item If $\det (M(R, C)\res{\gr{v, w_0}})^1_a\cdot\det (M(R, C)\res{\gr{v, w_0}})^d_b\neq 0$, it has sign $-\sigma$.
    \item If $\det (M(R, C)\res{\gr{v, w_0}})^{1, d}_{a, b}\neq 0$, it has sign $\sigma \cdot (-1)^{\ell(v)}$.
\end{enumerate}
\end{prop}

\begin{proof}
This follows from the proof of \cite[Theorem 4.18]{CSB}.
\end{proof}

We can now determine the sign of $\det X(R, C)\res{\gr{v, w_0}}$ on $k$-positive matrices.

\begin{thm} \label{thm:detSign}
Let $v \in S_n$ avoid $1324$ and $2143$ and let $k$ be the size of the largest square in $\gr{v,w_0}$. Choose $R, C \in \multichoose{[m]}{n}$. For $M$ a $k$-positive $m \times m$ matrix, $$(-1)^{\ell(v)} \det M(R,C)\res{\gr{v,w_0}}\geq 0$$
and equality holds if and only if $\gr{v, w_0}$ is not $(R,C)$-admissible.
\end{thm}

\begin{proof}
First, if $\gr{v, w_0}$ is not $(R,C)$-admissible, the determinant in question is obviously zero. So we assume $\gr{v, w_0}$ is $(R,C)$-admissible.

We follow the proof of \cite[Theorem 4.18]{CSB}, and proceed by induction on $n$. If $\gr{v, w_0}$ is a partition, a complement of a partition, or block-antidiagonal, we are done by \cref{cor:partitionimm}, \cref{cor:partitioncompimm}, or \cref{prop:ImmMultipleBlock}, respectively.

So we may assume that $v$ has at least 2 bounding boxes and that adjacent bounding boxes have nonempty intersection (where bounding boxes are ordered as usual by the row of their northeast corner). Because $v$ avoids 1324 and 2143, the final two bounding boxes of $\gr{v,w_0}$ are of opposite color by \cref{prop:alternatingBoxesAnti}. Without loss of generality, we assume the final box is red and the second to last box is blue. Otherwise, we can consider the antidiagonal transpose of $M(R,C)\res{\gr{v, w_0}}$. This is equal to $(\dot{w_0} M^T \dot{w_0})(\overline{C}, \overline{R})\res{\gr{w_0 v^{-1} w_0,w_0}}$ (using the notation in the proof of \cref{cor:partitioncompdet}) and has the same determinant as $M(R,C)\res{\gr{v, w_0}}$.

This means the final box is $\bbox{n, v_n}$, and the second to last box is $\bbox{a, v_a}$ for some $a<n$ with $1<v_a<v_n$. We analyze the sign of $\det M(R,C)\res{\gr{v, w_0}}$ using Lewis Carroll's identity on rows $a, b:=v^{-1}(1)$ and columns $1, d:=v_a$. Note that $a<b$ and $1<d$.

The proof of \cite[Theorem 4.18]{CSB} shows that each of the 5 known determinants in this Lewis Carroll's identity is equal to $\det M(R',C')\res{\gr{v', w_0}}$ for an appropriate choice of multisets $R', C'$ and permutation $v'$. We first show that two of these determinants, forming a single term on the right-hand side of the identity, are non-zero.

\begin{enumerate}
       
      
      \item Consider $(M(R,C)\res{\gr{v,w_0}})^1_b$. By \cref{prop:deleteDotDetAnti}, the determinant of this matrix is equal to the determinant of $M(R', C')\res{\gr{y,w_0}}$, where $y$ is obtained from $v$ by deleting 1 from $v$ in one-line notation and shifting appropriately, $R'=R \setminus \{r_b\}$ and $C'=C \setminus \{r_1\}$. 
      
      We will check that $\gr{y, w_0}$ is $(R', C')$-admissible. Note that because $(1, b)$ is not a spanning corner of $\gr{v, w_0}$, $\gr{y, w_0}=\gr{v, w_0}^1_b$ by \cref{prop:deleteDotDetAnti}. So we first check that removing column $1$ and row $b$ from $\gr{v, w_0}$ does not create any rows $i, j$ with both the same support and the same labels. By \cite[Theorem 4.18, pf. of (2)]{CSB}, rows $b, \dots, n$ of $\gr{v, w_0}$ all have support $\{1, \dots, v_n\}$. Note that removing column $1$ from $\gr{v, w_0}$ shortens rows $b, \dots, n$ by one and does not effect other rows, so it suffices to check that rows $b-1, \dots, n-1$ in $\gr{y, w_0}$ have distinct labels. Since $\gr{v, w_0}$ is $(R, C)$-admissible and rows $b, \dots, n$ of $\gr{v, w_0}$ have the same support, we must have $r_{b-1}\leq r_b<r_{b+1}< \cdots<r_n$. So, letting $r'_i$ denote the elements of $R'$, indexed in increasing order, we have $r'_{b-1}<r'_b <\cdots <r'_{n-1}$. 
      
      We now show there are no columns in $\gr{y, w_0}$ with both the same support and same labels. Columns $1, \dots, v_n$ of $\gr{v, w_0}$ have support containing $[b, n]$, and columns $v_n+1, \dots n$ have support contained in $[1, b-1]$. Removing row $b$ removes one element from the support of columns $1, \dots, v_n$ and does not effect other columns. Any two columns with the same support in $\gr{y, w_0}$ correspond to two columns with the same support in $\gr{v, w_0}$, and thus have different labels by the $(R, C)$-admissibility of $\gr{v, w_0}$.
      
      \item Consider $(M(R, C)\res{\gr{v,w_0}})^d_a$. By \cref{prop:deleteDotDetAnti}, the determinant of this matrix is equal to the determinant of $M(R', C')^d_a\res{\gr{z,w_0}}$, where $z$ is obtained from $v$ by deleting $v_a$ from $v$ in one-line notation and shifting appropriately, $R'=R\setminus\{r_a\}$, and $C'=C\setminus \{c_d\}$. See Figure~\ref{fig:deleteDotDetEx} for an example.
      
      \begin{figure}
          \centering
          \includegraphics[width=0.8\textwidth]{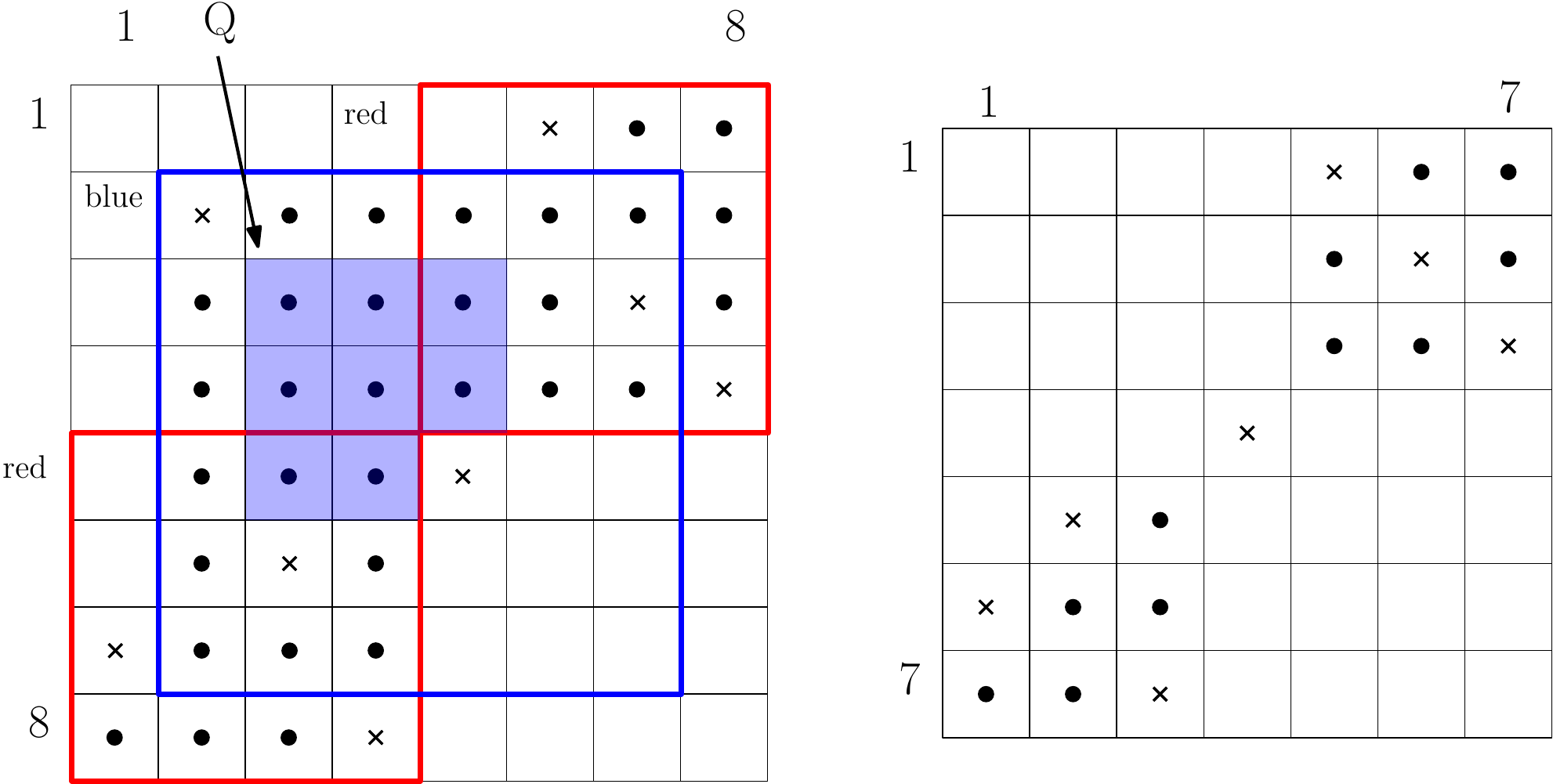}
          \caption{On the left, $\gr{v,w_0}$ for $v=62785314$. Elements of $\gr{v}$ are marked with crosses. On the right, $\gr{z,w_0}$ where $z=5674213$ is the permutation obtained by deleting 2 from the one-line notation of $v$ and shifting remaining numbers appropriately. Note that $\gr{z,w_0}$ is obtained from $\gr{v, w_0}$ by deleting row 2, column 2, and the shaded region $Q$, consisting of elements sandwiched only by non-inversions of the form $\lrangle{2, i}$.}
          \label{fig:deleteDotDetEx}
      \end{figure}
      
      As $(a, v_a)$ is a spanning corner of $\gr{v, w_0}$, $\gr{z, w_0}$ is obtained from $\gr{v, w_0}$ by deleting row $a$, column $d$, and the subset $Q \subset [n]^2$ consisting of all elements $(p, q)$ which are sandwiched only by a non-inversion of the form $\lrangle{a, i}$ (see Figure~\ref{fig:deleteDotDetEx}). Note that if $(p, q) \in Q$, then row $p$ of $\gr{v, w_0}$ has support $\{d, d+1, \dots, d+j\}$ for some $j$ and column $q$ of $\gr{v, w_0}$ has support $\{a, a+1, \dots, a+\ell\}$ for some $\ell$. Notice also that $Q$ consists of some initial chunk of each row and column of $\gr{v, w_0}$ it intersects; thus, deleting elements of $Q$ will not change the largest number in the support of any row or column. Since all corners of $\gr{v, w_0}$ are elements of $\gr{v}$ and $(a, d) \in \gr{v}$, there are no other corners in row $a$ or column $d$. So $a$ (resp. $d$) cannot be the largest element in the support of a column (resp. row). So for row $p$ in $\gr{v,w_0}_a^d$, with $\ell$ the largest element in the support of $p$,  $\delta^{-1}_d(\ell)$ is the largest element in the support of row $\delta^{-1}_a(p)$ in $\gr{v,w_0}$. An analogous statement holds for column $q$ in $\gr{v,w_0}_a^d$.
      
      Consider rows $p<p'$ of $\gr{v, w_0}$ with $r_p=r_p'$ and $p, p' \neq a$. Because $R$ is listed in weakly increasing order, $r_p=r_{p+1}=\cdots=r_{p'}$. By assumption, the support of these rows in $\gr{v, w_0}$ must be different. Suppose rows $s=\delta_{a}(p)$, $s'=\delta_a(p')$ have the same support in $\gr{z, w_0}$; say $\ell$ is the largest number in their support. The reasoning in the above paragraph implies that $\delta^{-1}_d(\ell)$ is the largest number in the support of rows $p, p'$ in $\gr{v, w_0}$, and thus also in rows $p, p+1, \dots, p'-1, p'$. So the smallest number in the support of rows $p,p+1, \dots, p'$ must be different. On the other hand, after deleting column $d$ and the elements of $Q$, the supports should be the same. These deletions remove the first element of a row only if that first element is in column $d$. Putting these together, we must have $p=a-1$, $p'=a+1$, and row $a+1$ has support starting at $d$; otherwise we obtain rows of $\gr{v, w_0}$ with the same label and same support. But now row $a$ is among rows $p, p+1, p'$, and rows $a$ and $p'=a+1$ have support starting at $d$, a contradiction.
      
      An identical argument with columns in place of rows shows that no two columns of $\gr{z, w_0}$ have the same support and the same label. So $\gr{z, w_0}$ is $(R', C')$-admissible.
\end{enumerate}

So by the inductive hypothesis, one term on the right-hand side of the identity is nonzero. Let $\sigma$ denote the sign of this term. By \cref{prop:sgn}, the other term on the right-hand side has sign $-\sigma$ if it is nonzero. In either case, the right-hand side has sign $\sigma$, and in particular is nonzero. Thus, both determinants on the left-hand side are non-zero. By \cref{prop:sgn}, the determinant $\det (M(R, C)\res{\gr{v, w_0}})^{1, d}_{a, b}$ has sign $\sigma \cdot (-1)^{\ell(v)}$, so dividing through by that determinant shows that $\det M(R, C)\res{\gr{v, w_0}}$ has sign $\ell(v)$.


\end{proof}

Taking this theorem with \cref{lem:dualCanonicalDet}, we can now prove \cref{thm:main-sq}.

\begin{proof}[Proof of \cref{thm:main-sq}]
By \cref{lem:dualCanonicalDet}, $$\Imm_v M(R, C)= (-1)^{\ell(v)} \det M(R, C)\res{\gr{v, w_0}}.$$

Let $k'\leq k$ be the size of the largest square in $\gr{v,w_0}$. By \cref{thm:detSign}, for $M$ $k'$-positive, the right hand side of this expression is positive. Any $k$-positive matrix is also $k'$-positive, so we are done. 
\end{proof}

\section{Future Directions}\label{sec:future}

The results in \cite{CSB} and this paper were inspired by the following conjecture of Pylyavskyy.

\begin{conj}[\hspace{1pt}\cite{Pyl}]\label{conj:Pasha}
Let $0<k<n$ be an integer and let $v \in S_n$ avoid the pattern $12\cdots (k+1)$. 
Then $\Imm_v X$ is $k$-positive. 
\end{conj}

This conjecture remains open. The relation between pattern avoidance and $k$-positivity of immanants is an interesting direction of further inquiry.

The results of this paper showcase an interesting phenomenon: the behavior of the dual canonical basis element $\Imm_v X(R, C)$ on $k$-positive matrices is the same as the behavior of the usual Kazhdan-Lusztig immanant $\Imm_v X$. Based on this, we make the following conjecture. 

\begin{conj}\label{conj:signControl}
Suppose $\Imm_v X$ is $k$-positive. Then as long as $\Imm_v X(R, C)$ is not identically zero, $\Imm_v X(R, C)$ is $k$-positive.
\end{conj}

We also make a related conjecture based on the same phenomenon, which is something of an intermediate conjecture; it would imply Conjecture~\ref{conj:Pasha} and would be implied by Conjectures~\ref{conj:Pasha} and \ref{conj:signControl} together.

\begin{conj}\label{conj:PashaStrong}
Let $0<k<n \leq m$ be integers and let $v \in S_n$ avoid the pattern $12\cdots (k+1)$. Let $R, C \in \multichoose{[m]}{n}$. If $\Imm_v X(R, C)$ is not identically zero, then $\Imm_v X(R, C)$ is $k$-positive. 
\end{conj}

The compact determinantal formulas we give for certain dual canonical basis elements may be useful to understand the relationship between the dual canonical basis of $\CC[SL_m]$ and its cluster algebra structure. Technically, the cluster algebra in question is the coordinate ring of $G^{w_0, w_0}$, the open double Bruhat cell in $SL_m$; $\CC[G^{w_0, w_0}]$ differs from $\CC[SL_m]$ by localization at certain principal minors. The cluster monomials of $\CC[G^{w_0, w_0}]$ are expected to be dual canonical basis elements. One natural question is: do the cluster monomials include the functions $\Imm_v X(R, C)$, where $v$ avoids 2143 and 1324? If so, can the $k$-positivity of these immanants be explained from a cluster algebraic viewpoint?

Work related to these questions appeared in the manuscript \cite{CJS}; the connection to Kazhdan-Lusztig immanants is explained in~\cite[Section 3.3]{Chepuri}. The results of \cite{CJS} show that $\Imm_v X(R, C)$ is a cluster variable for $v$ avoiding 123, 2143, 1432, and 3214. The immanants occurring in \cite{CJS} have a determinantal form given by \cref{lem:dualCanonicalDet}; they further conjecture that all cluster variables of $\CC[G^{w_0,w_0}]$ can be written as $\pm \det X(R, C)\res{P}$ for some $P \subset [n^2]$. Conjecturally, the Kazhdan--Lusztig immanants that can be written as $\pm \det X(R, C)\res{P}$ are the exactly $\Imm_v X(R, C)$ where $v$ is 2143 and 1324 avoiding. This leads to the following conjecture.

\begin{conj}\label{conj:cluster} Fix $m$ and let $G^{w_0, w_0}$ denote the big open double Bruhat cell in $SL_m$.
\begin{enumerate}
    \item All cluster variables of $\CC[G^{w_0,w_0}]$ are of the form $\Imm_v X(R, C)$ for some $v$ avoiding $2143$ and $1324$.
    \item For $v \in S_n $ avoiding $2143$ and $1342$ and $R, C \in \multichoose{[m]}{n}$ with $\gr{v, w_0}$ $(R, C)$-admissible, $\Imm_v X(R, C)$ is a cluster variable in $\CC[G^{w_0,w_0}]$ if it is irreducible and a cluster monomial otherwise.
\end{enumerate}

\end{conj}


\section{Acknowledgements}

We would like to thank Pavlo Pylyavskyy for suggesting this topic to us. This material is based upon work supported by the National Science Foundation under Grant No. DMS-1439786 while the first author was in residence at the Institute for Computational and Experimental Research in Mathematics in Providence, RI, during the Spring 2021 semester. The second author was supported by an NSF Graduate Research Fellowship DGE-1752814.

\bibliographystyle{siam}
\bibliography{bibliography}

\end{document}